\documentclass[11pt,reqno]{amsart}

\usepackage[T1]{fontenc}
\usepackage[boldsans]{ccfonts}
\usepackage{eulervm}
\renewcommand{\mathbf}{\mathbold}
\usepackage{bbm}

\usepackage{enumerate, amsmath, amsthm, amsfonts, amssymb, 
mathrsfs, graphicx, paralist}
\usepackage[usenames, dvipsnames]{color}
\usepackage[margin=1in]{geometry} 
\usepackage{hyperref}
\usepackage{comment}

\usepackage{mathabx} 
\usepackage{tikz-cd}

\numberwithin{equation}{section}
\newtheorem{Theorem}[equation]{Theorem}
\newtheorem{Proposition}[equation]{Proposition}
\newtheorem{Lemma}[equation]{Lemma}
\newtheorem{Corollary}[equation]{Corollary}

\theoremstyle{definition}
\newtheorem{Remark}[equation]{Remark}

\newtheorem{eg}[equation]{Example}

\newtheorem{Definition}[equation]{Definition}

\newcommand{\cA}{\mathcal{A}}

\newcommand{\cB}{\mathcal{B}}

\newcommand{\cN}{\mathcal{N}}

\newcommand{\cT}{\mathcal{T}}

\newcommand{\cW}{\mathcal{W}}

\newcommand{\cX}{\mathcal{X}}

\newcommand{\CC}{\mathbb{C}}

\newcommand{\PP}{\mathbb{P}}

\newcommand{\RR}{\mathbb{R}}

\newcommand{\XX}{\mathbb{X}}

\newcommand{\ZZ}{\mathbb{Z}}

\renewcommand{\phi}{\varphi}
\renewcommand{\emptyset}{\varnothing}

\renewcommand{\tilde}[1]{\widetilde{#1}}

\makeatletter
\def\Ddots{\mathinner{\mkern1mu\raise\p@
\vbox{\kern7\p@\hbox{.}}\mkern2mu
\raise4\p@\hbox{.}\mkern2mu\raise7\p@\hbox{.}\mkern1mu}}
\makeatother

\DeclareMathOperator{\reg}{reg}

\newcommand{\Gr}{\mathrm{Gr}}

\newcommand{\suchthat}{\mid}

\newcommand{\terminology}[1]{{\bf #1}}

\newcommand{\Sch}{\mathrm{Sch}}
\newcommand{\Alg}{\mathrm{Alg}}
\newcommand{\Set}{\mathrm{Set}}
\newcommand{\Func}{\mathrm{Func}}

\newcommand{\thick}{\mathrm{thick}}

\newcommand{\kk}{\Bbbk}
\newcommand{\X}{\cX}
\newcommand{\W}{\cW}

\begin{document}

\title{Symplectic leaves for generalized affine Grassmannian slices}
\author{Dinakar Muthiah}
\address{Kavli Institute for the Physics and Mathematics of the Universe (WPI), The University of Tokyo Institutes for Advanced Study, The University of Tokyo, Kashiwa, Chiba 277-8583, Japan}
\email{dinakar.muthiah@ipmu.jp}
\author{Alex Weekes}
\address{A.~Weekes: Perimeter Institute for Theoretical Physics, Canada}
\email{aweekes@perimeterinstitute.ca}

\maketitle

\begin{abstract}
The generalized affine Grassmannian slices $\overline{\W}_\mu^\lambda$ are algebraic varieties introduced by Braverman, Finkelberg, and Nakajima in their study of Coulomb branches of $3d$ $\mathcal{N}=4$ quiver gauge theories.  We prove a conjecture of theirs by showing that the dense open subset $\W_\mu^\lambda \subseteq \overline{\W}_\mu^\lambda$ is smooth. An explicit decomposition of $\overline{\W}_\mu^\lambda$ into symplectic leaves follows as a corollary. Our argument works over an arbitrary ring and in particular implies that the complex points $\W_\mu^\lambda(\CC)$ are a smooth holomorphic symplectic manifold.
\end{abstract}

\section{Introduction}

Affine Grassmannian slices for a reductive group $G$ are objects of considerable interest. As transversal slices to spherical Schubert varieties, they capture information about singularities in the affine Grassmannian. By the geometric Satake correspondence, these singularities are known to carry representation theoretic information about the Langlands dual group of $G$. Additionally, they have a Poisson structure that quantizes to the truncated shifted Yangians \cite{Kamnitzer-Webster-Weekes-Yacobi}. Furthermore, they form a large class of conical symplectic singularities that do not admit symplectic resolutions in general. As such, they form an important test ground for ideas in symplectic algebraic geometry and representation theory.

Recently, Braverman, Finkelberg, and Nakajima \cite{Braverman-Finkelberg-Nakajima} showed that affine Grassmannian slices arise as Coulomb branches of $3d$ $\cN=4$ quiver gauge theories. Their construction of affine Grassmannian slices is particularly satisfying because: (1) the quantization comes essentially for free in their construction, (2) their construction works for arbitrary symmetric Kac-Moody type. Because of point (2), the Coulomb branch perspective seems to be a fruitful path toward understanding the geometric Satake correspondence in affine type and beyond (see e.g. \cite{Finkelberg,Nakajima}).

However, their construction produces more than just affine Grassmannian slices: usual affine Grassmannian slices are indexed by a pair of dominant coweights $\lambda$ and $\mu$, but their construction does not constrain $\mu$ to be dominant. Rather, they construct \terminology{generalized affine Grassmannian slices} denoted $\overline{\W}^\lambda_\mu$ for $\lambda$ constrained to be dominant but for \emph{arbitrary} $\mu \leq \lambda$.

The geometry of the generalized affine Grassmannian slices for $\mu$ non-dominant is less understood than the case of $\mu$ dominant. For example, there is a disjoint decomposition 
\begin{align}
  \label{eq:16}
 \overline{\W}^\lambda_\mu = \bigsqcup_{\substack{ \nu \text{ dominant}, \\ \mu \leq \nu \leq \lambda} } \W^\nu_\mu
\end{align}
that Braverman, Finkelberg and Nakajima conjecture (\cite[Remark 3.19]{Braverman-Finkelberg-Nakajima}) to be a decomposition of $\overline{\W}^\lambda_\mu$ into symplectic leaves. They show that this would follow if one could show that the subvarieties $\W^\lambda_\mu$ are smooth for all $\lambda$ and $\mu$. In this note, we show the following, which proves their conjecture.
\begin{Theorem}[Corollary \ref{cor:main-result} in the main text]
  \label{thm:main-result-in-intro}
 For any $\lambda \geq \mu$ with $\lambda$ dominant, the variety $\W^\lambda_\mu$ is smooth.
\end{Theorem}

In particular, it follows that the set of $\CC$--points $\W_\mu^\lambda(\CC)$ is a smooth holomorphic symplectic manifold.  This verifies part of an expectation that it is also hyper-K\"ahler, since $\W_\mu^\lambda(\CC)$ should be identified with a moduli space of singular instantons on $\RR^3$, see \cite{Braverman-Finkelberg-Nakajima}, \cite{Bullimore-Dimofte-Gaiotto}.

\subsection{Previously known cases}

Theorem \ref{thm:main-result-in-intro} is known when $\mu$ is dominant because in this case $\overline{\W}^\lambda_\mu$ is a usual affine Grassmannian slice. It is also known for $\mu \leq w_0(\lambda)$ where $w_0$ is the longest element of the Weyl group \cite[Remark 3.19]{Braverman-Finkelberg-Nakajima}.   In type A, all cases are known by work of Nakajima and Takayama on Cherkis bow varieties \cite[Theorem 7.13]{Nakajima-Takayama}.  In work to appear, Krylov and Perunov prove Theorem \ref{thm:main-result-in-intro} in general type for $\lambda$ {miniscule}.  In fact, they prove more: they show that $\overline{\W}^\lambda_\mu = \W^\lambda_\mu$ is an affine space in this case.

We note that our main theorem has been expected by physicists, since $\W_\mu^\lambda$ should be a space of singular instantons as mentioned above, and that the decomposition (\ref{eq:16}) is an instance of {\em monopole bubbling}.  We refer the reader again to \cite{Bullimore-Dimofte-Gaiotto}, and to the references in the physics literature given in the introduction of \cite{Braverman-Finkelberg-Nakajima}, as well as in \cite{Nakajima-15}.

Finally, generalized affine Grassmannian slices of the form $\W^0_\mu$ had been previously considered in a different guise. These are the so called ``open Zastava spaces'' whose smoothness is deduced by a certain cohomology vanishing computation \cite{Finkelberg-Mirkovic}. Our approach gives a direct group-theoretic proof of this smoothness. It would be interested to understand how these two approaches are precisely related. We elaborate on this briefly in \S \ref{sec:open-zastava}.

\subsection{Our approach}

There is a group theoretic construction of generalized affine Grassmannian $\overline{\W}^\lambda_\mu$ and their open subschemes $\W^\lambda_\mu$ given in \cite[Section 2(xi)]{Braverman-Finkelberg-Nakajima}, inspired by the scattering matrix description of singular monopoles from \cite[Section 6.4.1]{Bullimore-Dimofte-Gaiotto}.   We slightly generalize this group-theoretic construction to produce spaces that we call $\overline{\X}^\lambda_\mu$ and $\X^\lambda_\mu$. We show these spaces are products of the corresponding $\W$-versions and an infinite dimensional affine space (Proposition \ref{prop:x-lambda-mu-product-w-lambda-mu}). We then show that the spaces $\X^\lambda_\mu$ are formally smooth (Theorem \ref{thm:main-theorem}), from which we conclude that the spaces $\W^\lambda_\mu$ are formally smooth. Because the $\overline{\W}^\lambda_\mu$ (and hence $\W^\lambda_\mu$) are known to be finitely presented (see Proposition \ref{prop:finite-presentation-of-w-lambda-mu} for a direct group-theoretic proof), we conclude that $\W^\lambda_\mu$ is in fact smooth. A subtle point in our approach is that we make use of the formal smoothness of an ind-scheme $\X^\lambda$ that is \emph{not} smooth, so the use of infinite-dimensional spaces and formal smoothness appears essential in our approach (see Remark \ref{rem:X-lambda-not-smooth}).

We note that our approach to smoothness is analogous to a standard approach to the smoothness of usual affine Grassmannian slices (and in fact general Schubert slices for partial flag varieties, see e.g.\cite[\S 1.4]{Kazhdan-Lusztig})). Our space $\X^\lambda_\mu$ is constructed using an auxiliary space $\X_\mu$ that plays the role of an open chart in the affine Grassmannian. We explain this briefly in \S \ref{sec:comparison-to-smoothness-for-open-affine-grassmannian-slices}.

\subsection{Acknowledgements}
We thank Michael Finkelberg, Hiraku Nakajima, and Oded Yacobi for helpful comments.
A.W.~is grateful to Kavli IPMU for hosting him during the workshop ``Representation theory, gauge theory, and integrable systems'', where this project was started.
This research was supported in part by Perimeter Institute for Theoretical Physics. Research at Perimeter Institute is supported by the Government of Canada through the Department of Innovation, Science and Economic Development Canada and by the Province of Ontario through the Ministry of Economic Development, Job Creation and Trade.

\section{Preliminaries}

\subsection{Schemes and functors}
Let $\kk$ be a commutative ring. Let $\Alg_\kk$ be the category of commutative $\kk$-algebras, and let $\Sch_\kk$ be the category of $\kk$-schemes. We define the category $\Func_{\kk}$ of $\kk$-functors to be the category of functors $\Alg_\kk \rightarrow \Set$. Recall that there is a fully-faithful embedding $\Sch_\kk \hookrightarrow \Func_\kk$ coming from the Yoneda lemma and the fact that morphisms of schemes are local for the Zariski topology. Following the usual terminology, we will call this inclusion the map sending a scheme to the functor it represents. All the functors we consider will be ind-schemes (of possibly ind-infinite type), so it is not strictly necessary to consider them as functors. However, we will be focused on questions of formal smoothness, so the functorial viewpoint will be essential.

\subsection{Formal smoothness}

Let $\phi: \tilde{A} \rightarrow A$ be a morphism in $\Alg_\kk$. Recall that we say $\phi$ is a \terminology{square-zero extension} if $\phi$ is surjective and the ideal $I = \ker(\phi)$ satisfies $I^2=0$. 

Let $X \in \Func_\kk$.  We say that $X$ is \terminology{formally smooth} if for every square zero extension $\tilde{A} \rightarrow A$, the map $X(\tilde{A}) \rightarrow X(A)$ is surjective. The relevance of formal smoothness is the following theorem of Grothendieck (see e.g. \cite[\href{https://stacks.math.columbia.edu/tag/02H6}{Lemma 02H6}]{stacks-project}).

\begin{Proposition}
\label{prop: smooth and fsmooth}
  Let $X$ be a finitely presented $\kk$-scheme. Then $X$ is smooth if and only if it is formally smooth.
\end{Proposition}

We record the following lemma for use later.
\begin{Lemma}
  \label{lem:formal-smoothness-and-products}
 Let $X,Y \in \Func_\kk$. Suppose $X \times Y$ is formally smooth and $X(\kk) \neq \emptyset$. Then $Y$ is formally smooth.
\end{Lemma}

\subsection{Group theoretic data}
Let $G$ be a connected split reductive group (see e.g.~\cite[Section II.1]{Jantzen} for an overview).  In particular, $G$ is defined over $\ZZ$. 
Let $T$ be a maximal torus, and let $U^+$ and $U^-$ be opposite unipotent subgroups (i.e. $U^-T$ and $U^+T$ are opposite Borel subgroups). Let $P$ be the coweight lattice of $T$, and let $Q$ be the coroot lattice.  We write $P_{++}$ for the dominant cone and write $Q_{+}$ for the positive coroot cone. Recall the dominance order where for $\lambda, \mu \in P$, we write $\mu \leq \lambda$ if $\lambda - \mu \in Q_+$.

Let $P^\vee$ be the weight lattice of $T$, and $P^\vee_{++}$ its cone of dominant weights.   For each $\Lambda \in P^\vee_{++}$, let $V(-\Lambda)$ be the Weyl module of $G$ with lowest weight ${-\Lambda}$.  This is a free $\kk$-module, and the lowest weight space $V(-\Lambda)_{-\Lambda}$ is a rank-one free $\kk$-module. Let $v_{-\Lambda}$ be a generator of this free module, and let $v^*_{-\Lambda}$ be the linear functional on $V(-\Lambda)$ that is equal to one on $v_{-\Lambda}$ and is zero on all other weight spaces. Let $\Delta_{\Lambda}$ be the regular function on $G$ defined by $\Delta_{\Lambda}(g) = \langle v_{-\Lambda}^*, g v_{-\Lambda} \rangle$.

Recall that the big cell of $G$ is the open subscheme $U^+ T U^- \subset G$. It is isomorphic to the product $U^+\times T \times U^-$, via the multiplication map.

\begin{Lemma}  Let $\Lambda, \Lambda' \in P^\vee_{++}$.  For $t\in T$ and $g\in U^+ T U^-$, we have:
  \label{lem:delta-i-and-torus}
  \begin{itemize}
\item[(a)]$\Delta_{\Lambda}(tg)= \Delta_{\Lambda}(g t) =  \Delta_{\Lambda}(t) \Delta_{\Lambda}(g)$,
\item[(b)]$\Delta_{\Lambda+\Lambda' }(g) = \Delta_{\Lambda}(g) \Delta_{\Lambda'}(g)$.
\end{itemize}
\end{Lemma}

The coordinate ring of $T$ is the group algebra of $P^\vee$, and thus for any $A \in \Alg_\kk$ there is a natural bijection $T(A) \longleftrightarrow \operatorname{Hom}_{\mathit{groups}}( P^\vee, A^{\times})$ where $A^\times \subset A$ are the units.  Since the Grothendieck group of the semigroup $P^\vee_{++}$ is canonically isomorphic to $P^\vee$, we have:

\begin{Lemma}
\label{lem: torus reconstruction}
For any $A \in \Alg_\kk$ there is a natural bijection
$$
T(A) \longleftrightarrow \operatorname{Hom}_{\textit{semigroups}}( P^\vee_{++}, A^{\times} ),
$$
such that $t \in T(A)$ corresponds to the homomorphism $\Lambda \mapsto \Delta_\Lambda(t)$.
\end{Lemma}

\subsection{Arcs and loops}
Let $z$ be a formal variable. Let $G[z]$ be the group object in $\kk$-functors defined by $G[z](A) = G(A[z])$ for each $A \in \Alg_\kk$. Similarly we define $G((z^{-1}))$ and $G[[z^{-1}]]$. We have closed embeddings $G[z] \hookrightarrow G((z^{-1}))$ and $G[[z^{-1}]] \hookrightarrow G((z^{-1}))$. Observe that $G[[z^{-1}]]$ is a scheme of infinite type, $G[z]$ is an ind-scheme of ind-finite type, and $G((z^{-1}))$ is an ind-scheme of ind-infinite type.

Let $\Gr_G^{\thick}$ be the thick affine Grassmannian, which is a scheme of infinite type. Let $1 \in
\Gr_G^{\thick}$ be the unit point. We have a transitive $G((z^{-1}))$-action on $\Gr_G^{\thick}$, and the stabilizer of the point $1$ is $G[z]$. Furthermore, it is easy to see that the map
\begin{align}
 G((z^{-1})) \rightarrow \Gr_G^{\thick} 
\end{align}
obtained by acting on $1 \in \Gr_G^{\thick}$ is a Zariski locally trivial principal bundle for the group $G[z]$. For this reason, one often writes $\Gr_G^{\thick} = G((z^{-1}))/G[z]$. However, one needs to be careful because the na\"ive functor indicated by $G((z^{-1}))/G[z]$ is not the functor that represents $\Gr_G^{\thick}$: one needs to appropriately sheafify.

\subsection{Generalized affine Grassmannian slices}

For $\lambda \in P$, we can consider the point $z^\lambda \in G((z^{-1}))$ and the corresponding point $z^\lambda \in \Gr_G^\thick$. We will restrict our attention to $\lambda \in P^{++}$. Then we write $\Gr^\lambda$ for the orbit inside of $\Gr_G^\thick$ of the point $z^\lambda$ under the group $G[z]$. Let $\overline{\Gr^\lambda}$ for the closure of the orbit (with its reduced scheme structure). It is well known that both $\Gr^\lambda$ and $\overline{\Gr^\lambda}$ are finite type schemes, and that $\Gr^\lambda$ is a smooth scheme.

We define the closed subfunctor $\overline{\X^\lambda}$ of $G((z^{-1}))$ to be the preimage of $\overline{\Gr^\lambda}$ under the map $G((z^{-1})) \rightarrow \Gr_G^{\thick}$ is surjective. We define $\X^\lambda$ to be the open subfunctor of $\overline{\X^\lambda}$ that is the primage of $\Gr^\lambda$. Suggestively, we will write $\X^\lambda = G[z] z^\lambda G[z]$ and $\overline{\X^\lambda} = \overline{G[z] z^\lambda G[z]}$. Observe that both $\overline{\X^\lambda}$ and $\X^\lambda$ have $G[z] \times G[z]$ actions given by left and right multiplications.

Let $G_1[[z^{-1}]]$ be the closed subscheme of $G[[z^{-1}]]$ consisting of elements that evaluate to $1 \in G$ modulo $z^{-1}$. Then we have the Gauss decomposition (in $\kk$-functors):
\begin{align}
  \label{eq:1}
  G_1[[z^{-1}]] = U_1^{+}[[z^{-1}]] \cdot T_1[[z^{-1}]] \cdot U_1^{-}[[z^{-1}]]  
\end{align}
where the factors on the right hand side are defined exactly as for $G$. For each $\mu \in P$, we will consider
\begin{align}
  \label{eq:2}
  \W_{\mu} = U_1^{+}[[z^{-1}]] \cdot z^{\mu} T_1[[z^{-1}]] \cdot U_1^{-}[[z^{-1}]]  
\end{align}
which is a closed sub-scheme of $G((z^{-1}))$. Note that this is also a product of $\kk$--functors.  We will restrict attention to $\mu$ with $\mu \leq \lambda$.

Following Braverman, Finkelberg, and Nakajima \cite[Section 2(xi)]{Braverman-Finkelberg-Nakajima}, we define:
\begin{align}
  \label{eq:3}
  \overline{\W}^{\lambda}_\mu = \overline{\X^{\lambda}} \cap \W_\mu
\end{align}
and 
\begin{align}
  \label{eq:4}
 \W^{\lambda}_\mu = \X^{\lambda} \cap \W_\mu
\end{align}
As they explain, $\overline{\W}^{\lambda}_\mu$ is a finite-type affine scheme with a Poisson structure. 
When $\mu$ is also dominant, they show that $\W^{\lambda}_\mu$ is isomorphic to a transversal slice in the affine Grassmannian (under the map $G((z^{-1})) \rightarrow \Gr_G^{\thick}$). For this reason, $\overline{\W^{\lambda}_\mu}$ is called a \terminology{generalized affine Grassmannian slice}. The following is \cite[Lemma 2.5]{Braverman-Finkelberg-Nakajima}; we record the following elementary proof for the benefit of the reader.

\begin{Proposition}
  \label{prop:finite-presentation-of-w-lambda-mu}
  For $\lambda$ and $\mu$ as above, $\overline{\W}^{\lambda}_\mu$ is a finitely presented affine $\kk$-scheme.
\end{Proposition}
\begin{proof}

Observe that $\overline{\W}^{\lambda}_\mu$ is defined over $\ZZ$, so it suffices to show that $\overline{\W}^{\lambda}_\mu$ is finite type affine scheme over $\ZZ$.
  
Embed $G$ as a closed subgroup of $GL_n$ such that $T$ and $U^{\pm}$ are compatible with the standard torus and unipotents in $GL_n$. We see that $\overline{\W}^{\lambda}_\mu$ is a closed subfunctor of a generalized affine Grassmannian slice for $GL_n$. Therefore, we are reduced to considering the case of $G = GL_n$.

Let $\overline{\XX^\lambda}$ be the closed subfunctor of $G((z^{-1}))$ consisting of elements $g$ such that for any $\Lambda \in P^\vee_{++}$, the matrix entries of $g$ acting on $V(-\Lambda)$ have poles in $z$ of order no worse than ${-\langle \lambda, \Lambda \rangle}$. Observe that $\overline{\X^\lambda}$ is a closed subfunctor of $\overline{\XX^\lambda}$. The functor $\overline{\XX^\lambda}$ is the preimage of the ``moduli version'' of $\overline{\Gr^\lambda}$ under the map $G((z^{-1})) \rightarrow \Gr^\thick_G$ (see \cite{Kamnitzer-Muthiah-Weekes} for a detailed discussion about the ``moduli version'' of $\overline{\Gr^\lambda}$). We will show that $\W_\mu \cap \overline{\XX^\lambda}$ is finite type, from which we conclude that $\W_\mu \cap \overline{\X^\lambda}$ is finite type. 

Let $g = x z^\mu t y \in \W_\mu$, i.e. $x \in U_1^{+}[[z^{-1}]]$, $t \in T_1[[z^{-1}]]$, and $y \in  U_1^{-}[[z^{-1}]]$.  Suppose further that $x z^\mu t y \in \overline{\XX^\lambda}$. As we are considering $GL_n$, $t$ is a diagonal matrix with diagonal entries $t_1, \ldots, t_n$, each of which is a series in $z^{-1}$ with constant term $1$,

For each $i$, let $a_i = t_i \cdots t_n$. Computing $\Delta_{\Lambda}(g)$ for $\Lambda$ the fundamental weights and the determinant character, we see that each $a_i$ is in fact a polynomial in $z^{-1}$ with an explicit bound on the degree coming from $\lambda$. Furthermore, the coefficients of each $t_i$ are polynomials in the finitely many coefficients of $a_1,\ldots,a_n$.

For integers $i,j$ with $1 \leq i < j \leq n$, consider $x_{ij}$, the $(i,j)$-th entry of the unipotent upper triangular matrix $x$. Each $x_{ij}$ is a power series in $z^{-1}$ with constant term $0$. If we act by $g$ on the lowest weight vector of the $(n-j+1)$-th fundamental representation, we see that $a_j x_{ij}$ appears as a coefficient. Therefore, $b_{ij}=a_j x_{ij}$ is a polynomial in $z^{-1}$ with an explicit bound on the degree, and the coefficients of $x_{ij}$ are polynomials in the finitely many coefficients of $a_j$ and $b_{ij}$.  A similar consideration applies to the matrix coefficients of $y$, and therefore $\W_\mu \cap \overline{\XX^\lambda}$ is a closed subscheme of a finite-dimensional affine space.

\end{proof}

\section{Main Result}

\subsection{The space $\X^\lambda_\mu$}

Let $\lambda$ and $\mu$ be as above. We will slightly generalize the definition of $\W_\mu$ and define
\begin{align}
  \label{eq:5}
  \X_\mu = U^{+}((z^{-1})) \cdot z^{\mu} T_1[[z^{-1}]] \cdot U^{-}((z^{-1}))  
\end{align}

\begin{Lemma}
  \label{lem:functor-of-X-mu}
  For any $A \in \Alg_\kk$, we have:
  \begin{align}
    \label{eq:11}
   \X_\mu(A) = \{ g \in G((z^{-1})) \suchthat \Delta_{\Lambda}(g) \in z^{-\langle \mu, \Lambda \rangle}\cdot \left(1 + z^{-1}A[[z^{-1}]] \right), \ \ \forall \Lambda \in P^\vee_{++} \}
  \end{align}
\end{Lemma}
\begin{proof}
It is easy to see that any $g\in \X_\mu(A)$ satisfies the conditions on $\Delta_\Lambda(g)$ above.  Conversely, if $g \in G( A((z^{-1})))$ satisfies these conditions, then in particular all $\Delta_\Lambda(g)$ are units in $A((z^{-1}))$.  This implies that $g$ is an $A((z^{-1}))$--point of the big cell $U^+ T U^- \subset G$, so we can write $g = u^+ t u^-$ with $u^\pm \in U^\pm( A((z^{-1})))$ and $t\in T( A((z^{-1})))$. In order to conclude that $g\in \W_\mu$, it remains to show that its factor $t$ has the correct form.  But since $\Delta_\Lambda(g) = \Delta_\Lambda(t)$, this follows from Lemma \ref{lem: torus reconstruction}.
\end{proof}

Consider the quotients $U^{\pm}((z^{-1}))/U^{\pm}_1[[z^{-1}]]$. Unlike the case of $G$, these quotients can be na\"ively interpreted because the natural maps 
\begin{align}
  \label{eq:6}
  U^{\pm}[z] \overset{\sim}{\rightarrow} U^{\pm}((z^{-1}))/U_1^{\pm}[[z^{-1}]] 
\end{align}
are isomorphisms. In particular, $U^{\pm}[z]$ gives us a section of the quotient maps $U^{\pm}((z^{-1})) \rightarrow U^{\pm}((z^{-1}))/U_1^{\pm}[[z^{-1}]]$.

We therefore obtain a map 
\begin{align}
  \label{eq:7}
  \pi_\mu : \X_\mu \rightarrow U^{+}[z] \times U^{-}[z]
\end{align}
given by sending a point $x z^\mu t y \in \X_\mu$ to $(x, y) \in U^{+}((z^{-1}))/U_1^{+}[[z^{-1}]] \times U^{-}((z^{-1}))/U_1^{-}[[z^{-1}]] \cong U^{+}[z] \times U^{-}[z]$.

\begin{Definition}
  \label{def:x-lambda-mu}
  Let $\lambda \in P^{++}$ and $\mu \in P$. Define:
  \begin{align}
    \label{eq:10}
\X^\lambda_\mu = \X^\lambda \cap \X_\mu    
  \end{align}
\end{Definition}

\begin{Proposition}
  \label{prop:x-lambda-mu-product-w-lambda-mu}
  The ind-scheme $\X^\lambda_\mu$ is isomorphic to the product $U^{+}[z] \times \W^\lambda_\mu \times U^{-}[z]$.
\end{Proposition}
\begin{proof}
  The map \eqref{eq:7} is our required map $\X^\lambda_\mu \rightarrow U^{+}[z] \times U^{-}[z]$. We need to construct a map $\X^\lambda_\mu \rightarrow \W^\lambda_\mu$. Observe that $\X^\lambda_\mu$ has a $U^{+}[z] \times U^{-}[z]$-action where the $U^{+}[z]$-factor acts by left multiplication, and the $U^{-}[z]$-factor acts by right multiplication. Furthermore,  the map $\pi_\mu$ is equivariant for this action. Given $g \in \X^\lambda_\mu$, consider $(x,y) = \pi_\mu(g)$. Then $x^{-1} \pi_\mu(g) y^{-1} \in \W^\lambda_\mu$. This defines our map $\X^\lambda_\mu \rightarrow \W^\lambda_\mu$, and it is clear that this and $\pi_\mu$ realize $\X^\lambda_\mu$ as the above product.
\end{proof}

\subsection{Formal smoothness of $\X^\lambda_\mu$}

\begin{Proposition}
  \label{prop:x-lambda-formally-smooth}
The ind-scheme $\X^\lambda$ is formally smooth.
\end{Proposition}

\begin{proof}
By definition, we have a Zariski-locally-trivial principal bundle $\X^\lambda  \rightarrow \Gr^\lambda$ for the group $G[z]$ over the smooth base $\Gr^\lambda$. So we are reduced to showing that the ind-scheme $G[z]$ is formally smooth.

 We need to show that for any square-zero extension $\tilde{A} \twoheadrightarrow A$, the map $G[z](\tilde{A}) \rightarrow G[z](A)$. This follows because $\tilde{A}[z] \twoheadrightarrow A[z]$ is also a square-zero extension and $G$ is smooth (and hence formally smooth).
\end{proof}

\begin{Remark}
  \label{rem:X-lambda-not-smooth}
  Although the ind-scheme is $\X^\lambda$ is formally smooth, it is \emph{not} smooth, i.e. it cannot be written as a an increasing union of smooth varieties (not even locally in the analytic topology). Fishel, Grojnowski, and Teleman show that thin affine Grassmannians are not smooth in this sense (even though they are formally smooth) \cite[Theorem 5.4]{Fishel-Grojnowski-Teleman}.  Locally in the Zariski topology, $\X^\lambda$ is isomorphic to a smooth variety times the big cell in the thin affine Grassmannian. Adapting the argument in \emph{loc. cit.}, we can conclude that $\X^\lambda$ is not smooth.

  Nonetheless, we will use the formal smoothness of $\X^\lambda$ to deduce the smoothness of the finitely presented scheme $\W^\lambda_\mu$ below. Because $\X^\lambda$ is not smooth, we cannot na\"ively truncate the argument to a finite-dimensional situation. This is a subtle point of our approach: the use of infinite dimensional ind-schemes and formal smoothness seems to be essential.
\end{Remark}

The following lemma is a slight variation of the classical Weierstra\ss\ Preparation Theorem (see e.g. \cite[Ch.~VII, \S 3.8]{Bourbaki}).

\begin{Lemma}
  \label{lem:weierstrass-preparation-type-statement}
  Let $\tilde{A} \rightarrow A$ be a square-zero extension with kernel $I$. Let $I[z]$ denote the set of polynomials in $\tilde{A}[z]$ having all coefficients lie in $I$. Let $D \in I[z] + 1 + z^{-1} \widetilde{A}[[z^{-1}]] $. Then there exists a unique polynomial $\gamma \in 1+ I[z]$ such that
  \begin{align}
    \label{eq:15}
\gamma D \in 1 + z^{-1} \widetilde{A}[[z^{-1}]]
  \end{align}
\end{Lemma}

\begin{proof}

For a series $s \in \widetilde{A}((z^{-1}))$, denote by $\reg (s) \in \widetilde{A}[z]$ its polynomial part.  
Write $D = x + 1 + a$ with $x\in  I[z]$ and $a\in z^{-1} \widetilde{A}[[z^{-1}]]$, and  consider a general element $\gamma = 1 + y \in 1 + I[z]$.  Then $xy =0$ since $I[z]^2 = 0$, and we have
\begin{equation}
\gamma D = x+1 +y +ya + a
\end{equation}
Therefore $\reg (\gamma D) =x + 1 + y +\reg( ya)$.  Since $a \in z^{-1}\widetilde{A}[[z^{-1}]]$, the coefficients of $y+ \reg(ya)$ have an upper-triangularity property with respect to the coefficients of $y$: for $n\geq 0$ the coefficient of $z^n$ in $y+\reg(ya)$ equals $y_n + \sum_{k > 0} a_{-k } y_{n+k}$, where $y_\ell, a_{\ell}$ denote the coefficients of $z^\ell$ in $y$ and $a$, respectively. Thus, by starting with the leading degree coefficient of $y$ and working downwards inductively, we can solve uniquely for $y$ such that $\reg(\gamma D) = 1$ (observe in particular that the degree of $y$ must be less than or equal to the degree of $x$).  
This proves the claim.
\end{proof}

\begin{Theorem}
  \label{thm:main-theorem}
 The ind-scheme $\X^\lambda_\mu$ is formally smooth.
\end{Theorem}

\begin{proof}
  Let $\tilde{A} \twoheadrightarrow A$ be a square-zero extension with kernel $I$, and let $g \in \X^\lambda_\mu(A) = \X^\lambda(A) \cap \X_\mu(A)$. Because $\X^\lambda$ is formally smooth, we can find $g' \in \X^\lambda(\tilde{A})$ lifting $g$. Let $\Lambda \in P^\vee_{++}$.  It may not be the case that $\Delta_\Lambda(g')$ lies in $z^{-\langle \mu, \Lambda \rangle}\cdot \left(1 + z^{-1}\tilde{A}[[z^{-1}]] \right)$, but we know at least that $\Delta_\Lambda(\tilde{g}) \in z^{-\langle \mu, \Lambda\rangle}\cdot \left(I[z] + 1 + z^{-1}\tilde{A}[[z^{-1}]] \right)$. Using Lemma \ref{lem:weierstrass-preparation-type-statement}, we can find a unique $\gamma_\Lambda \in 1 + I[z]$ so that
  \begin{align}
    \label{eq:12} \gamma_\Lambda \Delta_\Lambda(g') \in z^{-\langle \mu, \Lambda \rangle}\cdot \left(1 + z^{-1}\tilde{A}[[z^{-1}]] \right)
  \end{align}
  Since $\Delta_{\Lambda}(g) \Delta_{\Lambda'}(g) = \Delta_{\Lambda+\Lambda'}(g)$ by Lemma \ref{lem:delta-i-and-torus}(b), we must have $\gamma_{\Lambda+\Lambda'} = \gamma_{\Lambda} \gamma_{\Lambda'}$ by uniqueness. Also note that $1 + I[z] \subset \widetilde{A}[z]^\times$ since $I[z]^2 = 0$.  Therefore the map $\Lambda \mapsto \gamma_\Lambda$ defines an element of $\operatorname{Hom}_{\textrm{semigroups}}(P^\vee_{++}, \widetilde{A}[z]^\times)$, and thus a point $t \in T(\widetilde{A}[z])$ by Lemma \ref{lem: torus reconstruction}.   
  
  Define $\tilde{g} = t \cdot g'$. Because $t \in T(\tilde{A}[z])$ and $\X^\lambda$ is invariant under $G[z]$-multiplication, $\tilde{g} \in \X^\lambda(\tilde{A})$. By Lemma \ref{lem:delta-i-and-torus}(a) and Lemma \ref{lem: torus reconstruction}, for each $\Lambda \in P^\vee_{++}$, we have:
  \begin{align}
    \label{eq:14}
\Delta_\Lambda(\tilde{g}) = \Delta_\Lambda(t) \Delta_\Lambda(g') = \gamma_\Lambda \Delta_\Lambda(g')\in z^{-\langle \mu, \Lambda \rangle}\cdot \left(1 + z^{-1}\tilde{A}[[z^{-1}]] \right)
  \end{align}
Therefore $\tilde{g} \in \X^\lambda_\mu(\tilde{A})$ by Lemma \ref{lem:functor-of-X-mu}. Since $\tilde{g}$ is a lift of $g$, this shows that $\X^\lambda_\mu$ is formally smooth.
\end{proof}

Now, using Proposition \ref{prop:x-lambda-mu-product-w-lambda-mu}, Lemma \ref{lem:formal-smoothness-and-products}, and observing that $U^{+}[z] \times U^{-}[z]$ has a $\kk$-point, we can conclude that $\W^\lambda_\mu$ is formally smooth. Because $\W^\lambda_\mu$ is finitely-presented, smoothness and formal smoothness coincide by Proposition \ref{prop: smooth and fsmooth}. We conclude the following.
\begin{Corollary}
  \label{cor:main-result}
  The scheme $\W^\lambda_\mu$ is smooth.
\end{Corollary}

\subsection{Concluding remarks}

\subsubsection{Comparison to smoothness for open affine Grassmannian slices}
\label{sec:comparison-to-smoothness-for-open-affine-grassmannian-slices}

Our proof is inspired by the usual approach to showing smoothness of open affine Grassmannian slices. We will  quickly review this.  Let $\lambda$ and $\mu$ be dominant weights with $\mu \leq \lambda$. In this case, $\overline{\W}^\lambda_\mu$ is a closed subscheme of the affine Grassmannian with open subscheme $\W^\lambda_\mu$. Let $\cA_0$ denote the ``big cell'' of the affine Grassmannian. Left multiplying $\cA_0$ by $z^\mu$, we get an open subset $\cA_\mu$ of the affine Grassmannian that contains the point $z^\mu$.

Consider the intersection $\cA_\mu \cap \Gr^\lambda$, which is a smooth variety because it is an open subset of the smooth variety $\Gr^\lambda$. Furthermore, there is a map $\cA_\mu \cap \Gr^\lambda \rightarrow V$, where $V$ is the stabilizer inside of $U^+_1[z^{-1}]$ of the point $z^\mu$. Observe that $V$ is isomorphic to a finite dimensional affine space. This map realizes $\cA_\mu \cap \Gr^\lambda$ as $V \times \W^\lambda_\mu$. Therefore, we conclude that $\W^\lambda_\mu$ is smooth. We mention that this is a general calculation that works for arbitrary Schubert slices (see e.g.\cite[\S 1.4]{Kazhdan-Lusztig})).

For us, the space $\X_\mu$ plays the role of $\cA_\mu$, and $\X^\lambda$ plays the role of $\Gr^\lambda$. However, because $\X^\lambda$ is infinite-dimensional, smoothness is more subtle: hence our approach through formal smoothness.

\subsubsection{Open Zastava}
\label{sec:open-zastava}

In the case $\lambda=0$, we have $\overline{\W}^0_\mu = \W^0_\mu$, and the space $\W^0_\mu$ has been considered previously: it is precisely the ``open Zastava'' space consisting of degree $-\mu$ based maps from $\PP^1$ to the flag variety $\cB$ of $G$. In this case, smoothness was previously known by work of Finkelberg and Mirkovi\'c \cite{Finkelberg-Mirkovic}. Let $\phi : \PP^1 \rightarrow \cB$ be a degree $-\mu$ based map. They argue that $\phi$ is a smooth point of the open Zastava if and only if we have $H^1( \PP^1, \phi^{*} \cT_\cB) = 0$, where $\cT_\cB$ is the tangent sheaf of $\cT_\cB$. Because $\cT_\cB$ is globally generated, and all globally generated vector bundles on $\PP^1$ have vanishing higher cohomology, they deduce the necessary vanishing.

Our work gives another proof of the smoothness of open Zastava space. It would be very interesting to understand precisely how the two calculations correspond.

\bibliographystyle{amsalpha}
\bibliography{references}

\end{document}